\documentclass[a4paper]{article}

\usepackage{amsmath,amsfonts,amssymb,amsthm,amscd,enumerate}
\usepackage[all]{xy}
\usepackage[active]{srcltx}

\usepackage[dvipsnames]{color}

\newtheorem{thm}{Theorem}[section]
\newtheorem{prop}[thm]{Proposition}

\newtheorem{cor}[thm]{Corollary}

\theoremstyle{definition}
\newtheorem{defn}[thm]{Definition}

\newcommand{\nc}[2]{\newcommand{#1}{#2}}
\newcommand{\rnc}[2]{\renewcommand{#1}{#2}}
\rnc{\[}{\begin{equation}}
\rnc{\]}{\end{equation}}
\nc{\wegengruen}{\end{equation}}

\newcommand{\Z}{\mathbb{Z}}
\newcommand{\N}{\mathbb{N}}

\newcommand{\C}{\mathbb{C}}
\newcommand{\hsp}{{\hspace{-1pt}}}
\newcommand{\hs}{{\hspace{1pt}}}

\newcommand{\hH}{\mathcal{H}}
\newcommand{\cO}{\mathcal{O}}

\newcommand{\K}{{\mathcal{K}}}

\newcommand{\A}{{\mathcal{A}}}
\newcommand{\B}{\mathcal{B}}
\newcommand{\cC}{{\mathcal{C}}}

\newcommand{\M}{M}  

\newcommand{\Cq}{\cO(\C_q)}  
\newcommand{\dd}{\mathrm{d}}

\newcommand{\im}{\mathrm{i}}
\newcommand{\e}{\mathrm{e}}

\newcommand{\id}{{\mathrm{id}}}

\newcommand{\spec}{\mathrm{spec}}
\newcommand{\dom}{\mathrm{dom}}
\newcommand{\supp}{\mathrm{supp}}
\newcommand{\rmB}{B}  
\newcommand{\alg}{\text{*-}\mathrm{alg}}
\newcommand{\LL}{\mathcal{L}_2}
\newcommand{\tz}{{\tilde \zeta}}

\newcommand{\ra}{\rightarrow}
\newcommand{\lra}{\longrightarrow}
\newcommand*{\lhra}{\ensuremath{\lhook\joinrel\relbar\joinrel\rightarrow}}

\newcommand{\ip}[2]{\langle{#1},{#2}\rangle}

\title{A noncommutative 2-sphere generated by the quantum complex plane}

\author{{\sc Ismael Cohen} \\
\normalsize
Instituto de F\'isica y Matem\'aticas\\
\normalsize
Universidad Michoacana de San Nicol\'as de Hidalgo, Morelia, M\'exico\\[2pt]
\normalsize
and\\[2pt]
\normalsize
Centro de Ciencias Matem\'aticas, Campus Morelia\\
\normalsize
Universidad Nacional Aut\'onoma de M\'exico (UNAM), Morelia, M\'exico\\
\normalsize
e-mail: {\it ismaelcohen10@gmail.com}\\[16pt] 
{\sc Elmar Wagner\footnote{
corresponding author \ \  {\it MSC2010:} 46L65;  58B32. \ \  
}
} \\
\normalsize
Instituto de F\'isica y Matem\'aticas\\
\normalsize
Universidad Michoacana de San Nicol\'as de Hidalgo, Morelia, M\'exico\\
\normalsize
e-mail: {\it elmar@ifm.umich.mx}}

\date{}

\usepackage{geometry}
\begin{document}
\maketitle

\begin{abstract}
S. L. Woronowicz's theory of introducing C*-algebras generated by unbounded elements
is applied to $q$-normal operators satisfying the defining relation of the 
quantum complex plane.  The unique non-degenerate C*-algebra of bounded operators generated by 
a $q$-normal operator is computed and an abstract description is given by 
using crossed product algebras. If the spectrum of the modulus of the 
$q$-normal operator is the positive half line, this C*-algebra will be considered as the algebra of 
continuous functions on the quantum complex plane vanishing at infinity, and its 
unitization will be viewed as the algebra of 
continuous functions on a quantum 2-sphere. 
 \end{abstract}

\section{Introduction}                                 \label{sec:intro}
In his seminal paper \cite{Wo}, S.~L.~Woronowicz introduced the concept of C*-algebras 
generated by unbounded elements. His main motivation was to provide a proper mathematical 
framework for a topological  theory of non-compact quantum groups (cf.~\cite{PS,WoE2}). 
The basic idea is to use an affiliation relation to give a precise meaning to the statement that 
a finite set of (unbounded) operators satisfying certain relations generates a given C*-algebra. 
Naturally, one expects that a C*-algebra generated by unbounded operators is 
non-unital. This algebra is then viewed as the algebra of continuous functions vanishing at infinity on 
the underlying quantum space. 

Since compact spaces are better behaved than non-compact ones, 
we can pass to the compact case by adjoining a unit which corresponds to 
the one-point compactification of a locally compact space. 
The compactification may also turn a topologically trivial space into a non-trivial one,
as it is the case when considering for instance a sphere as the 
one-point compactification of the Euclidean plane. 
It is a non-commutative counterpart of this example that we want to study  
within Woronowicz's framework in the present paper. 

Our starting point is the quantum complex plane given as the complex *-algebra 
$\Cq$ generated by $\zeta$ and $\zeta^*$ satisfying the relation 
\[   												\label{zz*}
     \zeta \hs \zeta^* \,= \,q^2\hs \zeta^*\hs \zeta, 
\]
where, throughout the paper, we assume that $q\in(0,1)$.
A representation of $\Cq$ is given by a densely defined closed linear operator $\zeta$ 
on a separable Hilbert space satisfying~\eqref{zz*}. 
Such operators are known as $q$-normal operators  (or  $q^2$-normal operators). 
On the contrary to usual normal operators, non-zero $q$-normal operators are never 
bounded. In particular, they do not generate themselves a C*-algebra. 
This naturally  motivates the use of Woronowicz's theory for the 
study of the quantum complex plane $\Cq$ in the C*-algebra setting.  

The main result of this paper is an explicit description of the unique non-degenerate C*-algebra of 
bounded operators generated by a $q$-normal operator. However, since the 
algebra of polynomial functions on the quantum complex plane $\Cq$ is defined in a purely algebraic manner without 
referring to a Hilbert space, we prefer to give an (almost) Hilbert space free description and state our main theorem 
(Theorem \ref{thm}) in terms of crossed product C*-algebras. 
If the spectrum of the modulus $|\zeta|$ of the $q$-normal operator $\zeta$ is the positive half line $[0,\infty)$, the 
generated C*-algebra has an obvious interpretation as the algebra of continuous functions 
on the quantum complex plane vanishing at infinity. 
Its one-point compactification, given  by 
adjoining a unit,  will be viewed as the C*-algebra of continuous functions on 
a quantum 2-sphere.

\section{Representations of the quantum complex plane}                      \label{reps}
Although its main motivation comes from abstract 
algebras given by generators and relations, 
Woronowicz's framework starts by considering 
a set of unbounded operators on a Hilbert space 
affiliated with a C*-algebra of bounded operators \cite{Wo91}. 
For this reason, we are interested in operators on a separable Hilbert space 
satisfying the relation \eqref{zz*} in an appropriate sense. 
A natural choice is to assume 
$\zeta$ to be  $q$-normal. 
Such operators have been studied, e.\;g., in \cite{CSS,OS,Ota}.  
In this section, we will collect some facts about 
$q$-normal operators following closely 
the lines of \cite[Section~2]{CSS}. 

\begin{prop}[\cite{CSS,OS}]                       \label{qnormal}
Let $\zeta$ be a densely defined closed operator on a Hilbert space~$\hH$ and $\zeta = u\hs |\zeta|$ its polar
decomposition. Let $E$ denote the projection-valued measure on the Borel $\sigma$-algebra 
$\Sigma([0,\infty))$ such that $|\zeta|= \int \hsp\lambda\, \dd E(\lambda)$. 
Then the following statements are equivalent:
\begin{enumerate}[i)]
\item $\zeta$ is a $q$-normal operator, that is, $\zeta\zeta^* = q^2 \zeta^* \zeta$. 
\item $u\hs|\zeta|\hs u^* =q |\zeta|.$
\item $u \hs E(M)\hs u^* = E(q^{-1}M)$ for all $M\in \Sigma([0,\infty))$.  \label{iii}
\item $u\hspace{0.5pt}  f(|\zeta|) \hspace{0.5pt}  u^*\hsp =\hsp f(q|\zeta|)$ for every Borel function $f$ on $[0,\infty)$, 
where \label{iv}
$ f(|\zeta|)\hsp:=\!\int \! f(\lambda) \dd E(\lambda).
$ 
\end{enumerate}
\end{prop}

For a more explicit description of $q$-normal operators, note that 
$$
 \ker(\zeta)=\ker(\zeta^*)= \ker(|\zeta|)= E(\{0\})\hH. 
$$
Set 
\[ \label{Eq1}
 \hH_n:=E((q^{n+1}, q^n])\hH, \quad n\in\Z. 
\]
By Proposition \ref{qnormal} \ref{iii}), $u : \hH_n \ra \hH_{n-1}$ is an isomorphism. 
Therefore we can write 
\[
\hH_n = \{ h_n := u^{*n} h \,:\, h\in\hH_0 \}.  \label{Hn}
\]
Define $A: \hH_0 \ra \hH_0$ by 
$A:= \int_{(q,1]} \lambda\,\dd E(\lambda)$.  
Then, by Proposition \ref{qnormal} \ref{iv}), 
\[ \label{zhn}
\zeta\hs h_n= u\hs|\zeta|\hs u^{*n}\hs h=q^nu^{*(n-1)} A\hs h =  q^n(A\hs h)_{n-1} 
\]
on  $\ker({\zeta})^\perp= \underset{n\in\Z}{\oplus} \hH_n$. 
%
From the preceding, we get  the following description of $q$-normal operators. 
\begin{cor} \label{corz}
Let $\zeta$ be a non-zero $q$-normal operator on a Hilbert space $\hH$. 
Up to unitary equivalence, the action of $\zeta$ is determined by the 
following formulas: 
There exists a Hilbert space $\hH_0$ such that 
$\hH$ decomposes into the direct sum 
$\hH=\ker({\zeta})\oplus \underset{n\in\Z}{\oplus} \hH_n$ 
with $\hH_n = \hH_0$. 
For $h\in \hH_0$, let $h_n$ denote the vector in $\hH$ which has 
$h$ in the $n$-th component of the direct sum $\underset{n\in\Z}{\oplus} \hH_n$ 
and $0$ elsewhere. Then 
there exist a self-adjoint operator $A$ on $\hH_0$,  
satisfying that $\spec(A)\subset [q,1]$ with $q$ not being an eigenvalue,  such that  
$$
\zeta\hs h_n=  q^nA\hs h_{n-1} \quad \text{for all}\ \ h_n\in\hH_n.
$$
\end{cor}

It is well known \cite[Theorem VII.3]{RS} that each self-adjoint operator $T$ on a separable Hilbert space is 
unitarily equivalent to a direct sum of multiplication operators $\hat x$ on 
$\LL(\spec(T), \mu)$, \,where $(\hat{x}f)(x):= xf(x)$. 
We will apply this fact to the operators $|\zeta|$ and $A$ from the last corollary. 
Setting $\mu([0,\infty)\setminus \spec(|\zeta|))=0$, we may assume that 
$\hH=\LL([0,\infty), \mu)$. 
The spectral projections $E(M)$, $M\in\Sigma([0,\infty))$, from Proposition \ref{qnormal} are then given by 
multiplication with the indicator function $\chi_M$, that is, $(E(M)f)(x) =\chi_M(x)\hs f(x)$. 
Applying Proposition \ref{qnormal} \ref{iv}) to the Borel functions $\chi_M$ shows 
that $\mu$ is $q$-invariant:  $\mu(q M)= \mu (M)$ for all $M\in  \Sigma([0,\infty))$. 
As a consequence, $u: \hH \ra\hH$, \,$(uf)(x):=f(qx)$ is unitary. 
One easily checks that $\zeta:=u\hs\hat x$ defines a $q$-normal operator. 
It has been shown in  \cite{CSS} that any $q$-normal operator is unitarily 
equivalent to a direct sum of such operators. 
\begin{thm}[\cite{CSS}]   \label{zf}
Any $q$-normal operator is unitarily 
equivalent to a direct sum of operators of the following form: 
$\hH= \LL([0,\infty), \mu)$, where $\mu$ is a $q$-invariant Borel measure on 
$[0,\infty)$, $\dom(\zeta)= \{ f\in\hH\,:\,  \int_{[0,\infty)}  x^2\, |f(x)|^2\hs \dd\mu(x)< \infty\}$,  and 
\[ \label{L2rep}
(\zeta f)(x)=qxf(qx),\quad (\zeta^*f)(x)=xf(q^{-1}x)\quad \text{for all}\ \,f\in\dom(\zeta). 
\]
Moreover, for each $q$-invariant Borel measure $\mu$ on $[0,\infty)$, 
Equation \eqref{L2rep} defines a $q$-normal operator. 
\end{thm}

Note that a $q$-invariant measure  $\mu$ is uniquely determined by 
the value $\mu(\{0\})$ and 
its restriction 
$\mu_0:=\mu\!\!\upharpoonright_{ \Sigma([0,\infty))\cap (q,1]}$ via 
\[ \label{mu0}
\mu(M) = \sum_{k\in\Z} \mu_0 ( q^{-k}(M\cap (q^{k+1},q^k])) + \mu(M\cap\{0\} ).
\]
On the other hand, this formula defines for any measure $\mu_0$ on the Borel 
$\sigma$-algebra $\Sigma((q,1])$ 
a \hs$q$-invariant measure $\mu$ on $\Sigma([0,\infty))$. 
Once the value $\mu(\{0\})$ is fixed, this measure is unique. 

Suppose we are given a closed non-empty $q$-invariant set $ X\subset[0,\infty)$.  
Here, the $q$-in\-var\-ian\-ce means  $qX=X$. 
Since $K:=[q,1]\cap X$ is compact (and thus has a countable dense subset), 
there exists a finite Borel measure $\nu$ on $[q,1]$ such that 
$\supp(\nu)=K$, see~\cite{Ha}. Let $\delta_x$ denote the Dirac measure at $x\in[0,\infty)$ and 
define  $\mu_0$ to be the restriction of 
$\nu + \nu(\{q\})\delta_{1} + \nu(\{1\})\delta_{q}$ 
to $\Sigma([0,\infty))\cap (q,1]$. If $X=\{0\}$, set  $\mu(\{0\}):=1$. Then the measure $\mu$ determined  
by the formula in  \eqref{mu0} is a $q$-invariant $\sigma$-finite  Borel measure on 
$[0,\infty)$ such that $\supp(\mu)=X$. 

Recall that the multiplication operator $\hat x$ 
on $\LL([0,\infty), \mu)$ satisfies 
$\spec(\hat x) = \supp(\mu)$.  
Combining the discussion of the last paragraph with 
Theorem \ref{zf} gives the following corollary. 

\begin{cor} \label{cor}
For each non-empty $q$-invariant closed subset $X\subset[0,\infty)$, there exists 
a $q$-normal operator $\zeta$ such that $\spec(|\zeta|)=X$. 
\end{cor}

\section{C*-algebra generated by $q$-normal operators}                      \label{gen}
In this section, we will determine the C*-algebra generated by a 
$q$-normal operator. It turns out that this C*-algebra is closely 
related to crossed product C*-algebras. Since our transformation group 
will always be $\Z$, we use a slightly more direct 
(but equivalent) definition of 
crossed product C*-algebras than the usual one (cf.~\cite{Will}). 

Let $X\subset [0,\infty)$ be a closed non-empty $q$-invariant set.  
Then 
\[  \label{alpha}
 \alpha_q: C_0( X) \longrightarrow  C_0(X), \quad (\alpha_q(f))(x):= f(q x), 
\]
defines an automorphism of $C_0( X)$. Let $U$ be an abstract unitary element. 
Consider the *-algebra 
\[      \label{C0X} 
\alg\{C_0( X), U\}
:= \Big\{  \sum_{\text{finite}} f_k \hs U^k\,:\, f_k\in C_0( X),\ \, k\in\Z\Big\}, 
\]
with multiplication and involution determined by 
\[ \label{multinv}
fU^n g U^m = f\alpha_q^n(g)U^{n+m}, \quad 
(fU^n)^*= \alpha_q^{-n}(\bar f)\hs U^{-n}, \quad f,g \in C_0( X),\ \,n,m\in\Z,
\]
where $\bar f$ denotes the complex conjugate of $f$. 
Note that $U\notin \alg\{C_0( X), U\}$. 

We remark that the use of the unitary $U$ in the definition  of 
$\alg\{C_0( X), U\}$ in Equation \eqref{C0X} is superfluous since 
the way the functions $f_k$ and $g_j$ multiply in the product
$ (\underset{\text{finite}}{\sum} f_k \hs U^k)
(\underset{\text{finite}}{\sum} g_j \hs U^j)$ is known as the 
convolution product. However, later on, we will replace 
$U$ by a unitary operator $u$ and 
$f_k$ by the 
operator $f_k(|\zeta|)$, 
where $|\zeta|$ and $u$ are defined   
by the polar decomposition $\zeta=u\hs |\zeta|$ of an appropriate 
$q$-normal operator $\zeta$, so our notation is more suggestive.

For a Hilbert space $\hH$, we denote by $\rmB(\hH)$ the C*-algebra of 
bounded operators. 
A covariant representation of $\alg\{C_0( X), U\}$ 
is given by a unitary 
operator $V\in \rmB(\hH)$ and 
a *-representation $\pi_0 : C_0( X)\ra  \rmB(\hH)$ such that $V \pi_0(f)= \pi_0(\alpha_q(f)) V$ for all 
$f\in C_0( X)$.  Setting $\pi(U):=V$ and  $\pi(fU^n):= \pi_0(f) \pi(U)^n$, we obtain 
a *-representation $\pi : \alg\{C_0( X), U\}\ra  \rmB(\hH)$. 
Therefore we can view covariant representations as *-re\-pre\-sen\-ta\-tions of 
$\alg\{C_0( X), U\}$ satisfying $\pi(fU^n):= \pi_0(f) \pi(U)^n$ with $\pi(U)$ unitary. 
Now we define the crossed product C*-algebra $C_0( X)\rtimes \Z$ as 
the closure of $\alg\{C_0( X), U\}$ under the norm 
\[ \label{norm}
\| a \|_{\text{univ}} 
:= \sup \{ \|\pi(a)\|\,:\, \pi\ \text{is a covariant representation}\}, \quad 
a\in \alg\{C_0( X), U\}. 
\]  
The existence of the norm defined in \eqref{norm} follows from general considerations, see \cite{Will}. 

Note that $X{\setminus}\{0\}$ remains to be $q$-invariant for any $q$-invariant set $X\subset[0,\infty)$. 
The crossed product C*-algebra $C_0(X{\setminus}\{0\})\rtimes \Z$ is defined similarly 
to the above with $X$ replaced by $X{\setminus}\{0\}$. That is, 
$$
C_0(X\hsp\setminus\hsp\{0\})\rtimes \Z:= \|\cdot\|_{\text{univ}}\text{-}\mathrm{cls}
\Big\{  \sum_{\text{finite}} f_k \hs U^k\,:\, f_k\in C_0( X\hsp\setminus\hsp\{0\}),\ k\in\Z\Big\} , 
$$
where the multiplication and involution are determined by \eqref{multinv}, and the norm is 
given as in \eqref{norm}. 

Suppose that $\zeta=u|\zeta|$ is a $q$-normal operator with $\spec(|\zeta|)=X\neq \{0\}$. 
It follows from Proposition \ref{qnormal} \ref{iv}) that we obtain a covariant representation 
of $\alg\{C_0( X), U\}$ and $\alg\{C_0( X\hsp\setminus\hsp\{0\}), U\}$ 
on $\ker(|\zeta|)^\perp$ 
by setting 
$\pi_0(f)=f(|\zeta|)$ and $\pi(U)=u$. 
The next proposition shows that we can always think of the corresponding crossed product 
C*-algebras as the closure of the image of these covariant representations. 

In the proof of the proposition, we will need the following notation: 
Given a C*-algebra $\A$, we denote by $\M(\A)$ its multiplier C*-algebra (see e.g.\ \cite[Section 1.5]{Will}). 
Let $\hH$ be a Hilbert space. We say that a *-representation 
$\pi: \A\ra \rmB(\hH)$ is non-degenerate, if the set $\pi(\A)\hH$ is dense in $\hH$. 
It is known that  each non-degenerate $\pi$ admits a unique extension 
(denoted by the same symbol)  $\pi: \M(\A)\ra \rmB(\hH)$.

\begin{prop} \label{propX}
Let $X\neq \{ 0\}$ be a 
non-empty $q$-invariant closed subset of $[0,\infty)$. Then 
\begin{align*}
C_0( X)\rtimes \Z&\,\cong\, \B_1:=  \|\hsp\cdot\hsp\|\text{-}\mathrm{cls}
\Big\{  \sum_{\text{{\rm finite}}} f_k(|\zeta|) \hs u^k\,:\, f_k\in C_0( X),\ k\in\Z \Big\}, 
 \\
C_0(X\hsp\setminus\hsp\{0\})\rtimes \Z&\,\cong\, \B_0:=\|\hsp\cdot\hsp\|\text{-}\mathrm{cls}
\Big\{ \! \sum_{\text{{\rm finite}}}\hsp f_k(|\zeta|) \hs u^k : f_k\hsp\in\hsp C_0( X\hsp\setminus\hsp\{0\}),\;
k\hsp\in\hsp\Z\Big\} , 
\end{align*}
where $u$ and $|\zeta|$ are defined by the 
polar decomposition
$\zeta=u|\zeta|$ of a $q$-normal operator $\zeta$ such that $\spec(|\zeta|)= X$ and $\ker(|\zeta|)=\{0\}$. 
\end{prop}
\begin{proof}
First we remark that a $q$-normal operator with $\spec(|\zeta|)= X$ exists by Corollary \ref{cor}. 
Taking its restriction to $\ker(|\zeta|)^\perp$, we may assume that $\ker(|\zeta|)=\{0\}$.  
For brevity, set $X_0:= X\hsp\setminus\hsp\{0\}$ and $X_1:=X$. 
The proposition will be proven by invoking the universal property of crossed product C*-algebras. 
That is, if we show that 
\begin{enumerate}[(a)]
\item  \label{Ra}
there is covariant representation 
$\rho_i : \alg\{C_0( X_i), U\}\ra  \M(\B_i)$, 
\item  \label{Rb}
given a  covariant representation $\pi_i : \alg\{C_0( X_i), U\}\ra \rmB(\hH)$, 
there is a non-degenerate *-re\-pre\-sen\-ta\-tion $\Pi_i : \B_i\ra \rmB(\hH)$ such that its unique extension to 
 $\M(\B_i)$ satisfies  $\Pi_i\circ  \rho_i = \pi_i$, 
 \item  \label{Rc}
 $\B_i = \|\hsp\cdot\hsp\|\text{-}\mathrm{cls}\hs\{ \rho_i(f)\rho_i(U)^k \,:\, f\in C_0( X_i), \ k\in\Z\}$, 
\end{enumerate}
then $\B_i\cong C_0( X_i)\rtimes \Z$, \,$i=0,1$, by Raeburn's Theorem \cite[Theorem 2.61]{Will}.
Setting $\rho_i(f):=f(|\zeta|)$, \,$f\in C_0( X_i)$, and $\rho_i(U)=u$, the item 
(\ref{Ra}) follows from Proposition~\ref{qnormal}~\ref{iv}) and (c) 
from the definition of $\B_i $. 

To prove (\ref{Rb}), we show that 
$\Pi_i(\underset{\text{finite}}{\sum} f_k(|\zeta|) \hs u^k )
:=\underset{\text{finite}}{\sum}\pi_i(f_k)\hs \pi_i(U)^k$ is well defined.  
For this, it suffices to verify that $\sum_{k=M}^N f_k(|\zeta|) \hs u^k=0$ implies 
$f_k=0$ for all~$k$, where $f_k\in C_0( X_i)$,  $M,N\in\Z$ and  $M\leq N$. 
By unitary equivalence, we may assume that the action of $\zeta=u\hs |\zeta|$ is given by the formulas of 
Corollary \ref{corz}. Recall that $\hH_{n}\perp \hH_m$ for $n\neq m$, \,$f_k(|\zeta|) : \hH_n\ra \hH_n$ 
and  $u^k h_{n}=h_{n-k} \in\hH_{n-k}$. 
Suppose now that $f_M\neq 0$. 
Then there exist $n\in\Z$ and $g_n,\,h_n\in \hH_n$ such that $\ip{g_n}{f_M(|\zeta|)h_n}\neq 0$. 
Hence 
$$
\ip{g_n}{ \sum_{k=M}^N f_k(|\zeta|) \hs u^k \hs h_{n+M}} = \ip{g_n}{ \sum_{k=M}^N f_k(|\zeta|)  \hs h_{n+M-k}} = 
\ip{g_n}{f_M(|\zeta|)h_n}\neq 0. 
$$
Therefore $\sum_{k=M}^N f_k(|\zeta|) \hs u^k=0$ implies $f_M=0$. By induction on $m=M, M\hsp+\hsp 1, \ldots\hs$, 
we conclude that $f_m=0$ for all 
$m=M, \hs M\hsp +\hsp 1, \ldots, N$, so $\Pi_i$ is well defined. The non-degeneracy is 
easily shown by applying $\Pi_i$ to an approximate unit of $C_0(X_i)$. 
\end{proof}


Before stating our main theorem, we recall 
Woronowicz's definition \cite[Definition~3.1]{Wo} of a C*-algebra generated by unbounded elements. 
Let $\hH$ be a separable Hilbert space and let $\A\subset\rmB(\hH)$ be a C*-algebra. We say that a 
densely defined closed operator $T$ acting on $\hH$ is affiliated with $\A$ if its $z$-transform 
\[  \label{zetatrans}
z_T:=T(1+T^*T)^{-1/2}
\]
belongs to the multiplier algebra $\M(\A)$ and if the set $(1-z_T^* z_T)\A$ is dense in $\A$. 
Suppose that $\A$ is non-degenerate, i.e., the set $\A\hH$ is dense in $\hH$. Then, 
given a C*-algebra $\A_0$, the set of morphisms from $\A_0$ into $\A$ is defined as 
\begin{align*}
\mathrm{Mor}(\A_0,\A):=\{ \pi:\A_0\lra \M(\A)\subset \rmB(\hH)\,:\, \ &\pi\ \,\text{is a *-homomorphism}\\ 
&\text{and}\ \,\pi(\A_0)\A\ \,\text{is dense in}\ \A\}. 
\end{align*}
Let $\K$ be a separable Hilbert space and $\pi : \A\ra \rmB(\K)$ a non-degenerate representation. 
Using the fact that $\pi$ admits a unique extension 
$\pi: \M(\A)\ra \rmB(\K)$, we can define the $\pi$-image of an 
operator $T$ affiliated with $\A$ by 
\[   \label{pi-im}
\pi(T) := \pi(z_T)\big(1-\pi(z_T)^*\pi(z_T)\big)^{-1/2} . 
\]
Note that $z_{\pi(T)}= \pi(z_T)$ and that $\pi(T)$ is uniquely determined by $ \pi(z_T)$. 
Now, given a C*-algebra $\A$ and a finite set of elements $T_1,\ldots,T_N$ affiliated with $\A$, it is said 
that $T_1,\ldots,T_N$ generate $\A$ if for any non-degenerate representation $\pi : \A\ra \rmB(\K)$ 
and any non-degenerate C*-algebra $\B\subset\rmB(\K)$, one has 
\[ \label{Wodef}
\pi(T_1),\ldots ,\pi(T_N)\ \,\text{are affiliated with}\ \,\B \qquad 
\Longrightarrow     
\qquad \pi\in\mathrm{Mor}(\A,\B). 
\]
It follows immediately from \cite[Proposition 3.2]{Wo} 
that a non-degenerate C*-algebra $\A$ generated 
by $T_1,\ldots,T_N$ is unique; see also the comments 
before ibid. Theorem 4.2  for the more general case of a 
C*-algebra $\A$ generated by a quantum family of affiliated elements.

Our main goal, achieved in the next theorem, is to give an explicit description 
of the C*-algebra generated by a non-trivial $q$-normal operator.  
Since $q$-normal operators act on Hilbert spaces, it would be natural 
to search for a subalgebra of bounded operators. 
However, the use of crossed product algebras and 
the last proposition allow us to describe the generated C*-algebra 
without reference to a Hilbert space.

\begin{thm} \label{thm}
Let $\zeta$ be a $q$-normal operator on a separable Hilbert space $\hH$ 
such that $X:=\spec(|\zeta|)\neq \{0\}$. 
Then the unique non-degenerate C*-algebra in $\rmB(\hH)$ generated by~$\zeta$
is isomorphic to 
\[ \label{Czz}
C^*_0(\zeta,\zeta^*):=\|\cdot\|\text{-}\mathrm{cls}\hs \Big\{ 
\sum_{\text{{\rm finite}}} f_k \hs U^k\in C_0( X)\rtimes \Z \;:\; k\in\Z,\,\  f_k(0)=0\ \, 
\text{for all}\ \,k\neq 0 
\Big\}, 
\]
where the norm closure is taken in the C*-algebra $C_0( X)\rtimes \Z$.   
\end{thm}
\begin{proof}
The first step of the proof consists in 
identifying the abstractly defined C*-algebra $C^*_0(\zeta,\zeta^*)$ 
with a non-degenerate C*-subalgebra of $\A\subset\rmB(\hH)$.  
Let $\zeta=u\hs |\zeta|$ be the polar decomposition of $\zeta$.   
Up to unitary equivalence, we may assume that $\hH$ and  $\zeta$ are given by the formulas 
of Corollary~\ref{corz}. 
Then $\ker(\zeta)=\ker(|\zeta|)$ and  
$u \!\!\upharpoonright_{\ker(|\zeta|)^\perp}$ is a unitary operator. Furthermore, 
by Proposition~\ref{propX},
we have  $C^*_0(\zeta,\zeta^*)\cong \B\subset \rmB(\ker(|\zeta|)^\perp )$,    
where 
$$
\B := \|\hsp\cdot\hsp\|\text{-}\mathrm{cls}\hs
\Big\{  \sum_{\text{finite}} f_k(|\zeta|) \hs u^k{ \upharpoonright}_{\ker(|\zeta|)^\perp}  \,:\,   f_k\in C_0(\spec(|\zeta|)),\ \, f_k(0)=0\ \, 
\text{for all}\ \,k\neq 0   \Big\}. 
$$
On $\hH =  \ker(|\zeta|) \oplus \ker(|\zeta|)^\perp$, set 
\begin{align} \nonumber
\A := \|\cdot\|\text{-}\mathrm{cls} \hs\Big\{  
f_0(|\zeta|){ \upharpoonright}_{\ker(|\zeta|)} \oplus \sum_{\text{finite}} f_k(|\zeta|) \hs u^k{ \upharpoonright}_{\ker(|\zeta|)^\perp}  \,:\,\ 
 &f_k\in C_0(\spec(|\zeta|)) \ \,\text{and} \\[-8pt]
&f_k(0)=0\ \, \text{for all}\ \, k\neq 0   \Big\}.  \label{A}
\end{align}
Recall from Corollary \ref{corz} that $\hH =  \ker(|\zeta|) \oplus \underset{n\in\Z}{\oplus}\hH_n$, where $\hH_n\cong \hH_0$ 
is the image of the spectral projection of $|\zeta|$ corresponding to the Borel set $(q^{n+1},q^n]\cap \spec(|\zeta|)$. 
Since 
\begin{align*}
\| f_0(|\zeta|){ \upharpoonright}_{\ker(|\zeta|)}    \|=   |f_0(0)| &\,\leq\, \sup\{ |f_0(x)|\,:\, x\in \spec(|\zeta|){\setminus}\{0\}\}\\
             &\,=\, \sup\{ \|f_0(|\zeta|)h_n\|\,:\, h_n\in \hH_n,\ \,\|h_n\|=1,\ \,n\in\Z\} \\
             &\,\leq\,  \Big\|\sum_{\text{finite}} f_k(|\zeta|) \hs u^k{ \upharpoonright}_{\ker(|\zeta|)^\perp}  \Big\|  , 
\end{align*}
one easily checks  that $\Psi :\A\ra\B$, given by 
\[  \label{isopsi}
\Psi\Big(f_0(|\zeta|){ \upharpoonright}_{\ker(|\zeta|)} \oplus 
\sum_{\text{finite}} f_k(|\zeta|) \hs u^k{ \upharpoonright}_{\ker(|\zeta|)^\perp}\Big)= 
 \sum_{\text{finite}} f_k(|\zeta|) \hs u^k{ \upharpoonright}_{\ker(|\zeta|)^\perp},
 \]
 defines an isometric *-isomorphism. 
 Thus $C^*_0(\zeta,\zeta^*)\cong \A$. The non-degeneracy of $\A$ follows from the fact that 
 for each $m\in\N$, there exists a $\varphi_m\in C_0(\spec(|\zeta|))$ satisfying $\varphi_m(t)=1$ for all 
 $t\in [0,q^{-m})$ so that $\varphi_m(|\zeta|)\in\A$ acts as the identity on  
 $\ker(|\zeta|) \oplus \overset{\infty}{\underset{n=-m}{\oplus}}\hH_n$. 

The theorem will now be proven by applying \cite[Theorem 3.3]{Wo} 
to $\A\subset \rmB(\hH)$, 
so it suffices show that 
\begin{enumerate}[(a)]
\item  \label{Ta} 
$\zeta$ is affiliated with $\A$, 
\item  \label{Tb} 
$\zeta$ separates the representations of $\A$, 
\item  \label{Tc} $(1+\zeta^*\zeta)^{-1}\in\A$. 
\end{enumerate}
Observe that (\ref{Tc}) holds trivially since the function $f_0$, defined by $f_0(t):= (1+t^2)^{-1}$, 
belongs to $C_0(\spec(|\zeta|))$ and 
$(1+\zeta^*\zeta)^{-1}=f_0(|\zeta|) = f_0(|\zeta|){ \upharpoonright}_{\ker(|\zeta|)}\oplus f_0(|\zeta|){ \upharpoonright}_{\ker(|\zeta|)^\perp}$. 

Next we verify (\ref{Ta}). 
By the definition of the affiliation relation,   
this means that 
\[
z_\zeta:= \zeta(1+\zeta^*\zeta)^{-1/2} \in\M(\A),
\]
and that 
\[ \label{clsA}
\|\hsp\cdot\hsp\|\text{-}\mathrm{cls}\big\{ (1-z_\zeta^*\hs z_\zeta)a\,:\, a\in\A\big\} = \A. 
\]
By the formulas in Corollary \ref{corz},  
$z_\zeta = 0 \oplus  (u |\zeta| (1+|\zeta|^2)^{-1/2}){ \upharpoonright}_{\ker(|\zeta|)^\perp}$. 
On $\ker(|\zeta|)^\perp$, we get from Proposition \ref{qnormal} \ref{iv}) 
$u |\zeta| (1+|\zeta|^2)^{-1/2} f_k(|\zeta|) u^k=  q|\zeta| (1+q^2|\zeta|^2)^{-1/2} f_k(q|\zeta|) u^{k+1}$. 
Since the function $\tilde f_k$ given by $\tilde f_k(t):= \frac{qt}{ \sqrt{1+q^2t^2}} f_k(qt)$ belongs to $C_0(\spec(|\zeta|))$ 
for any $f_k\in C_0(\spec(|\zeta|))$ and satisfies $\tilde f_k(0)=0$, we see that multiplying from the left with $z_\zeta$ maps the 
defining set of $\A$ on the right-hand side of \eqref{A} into itself. Taking the closure yields $z_\zeta\in\M(\A)$. 

To show \eqref{clsA}, note that $1-z_\zeta^*\hs z_\zeta = (1+|\zeta|^2)^{-1}$. 
Let $\{\varphi_n\}_{ n\in\N}$ be an approximate unit for $C_0(\spec(|\zeta|))$ such that 
each $\varphi_n$ has compact support. Set $\phi_n(t):= (1+t^2)\varphi_n(t)$. 
Then $\phi_n\in C_0(\spec(|\zeta|))$ and 
$\underset{n\to\infty}{\lim}(1-z_\zeta^*\hs z_\zeta)\hs \phi_n(|\zeta|)\hs f(|\zeta|)=\underset{n\to\infty}{\lim}\varphi_n(|\zeta|) f(|\zeta|)=f(|\zeta|)$ 
for all $f\in C_0(\spec(|\zeta|))$. From this, one easily concludes that the closure on the left-hand side of \eqref{clsA} 
contains the defining set of $\A$ on the right-hand side of \eqref{A} which proves~\eqref{clsA}. 

We turn now to the proof of (\ref{Tb}). Since $\A$ and $\B$ are  isomorphic, it suffices to consider representations of $\B$. 
Set $\tilde \zeta:= \zeta { \upharpoonright}_{\ker(|\zeta|)^\perp} $. 
Note that $\spec(|\tilde \zeta| )=\spec(|\zeta|)$ by the $q$-invariance of 
 $\spec(|\zeta|)\neq \{0\}$. 
Therefore $\B$ is generated by 
$\underset{\text{finite}}{\sum} f_k(|\tilde \zeta|) \hs u^k$ with $f_k\in C_0(\spec(|\zeta|))$ and  $f_k(0)=0$ 
for all $k\neq 0$.  
The same arguments as above show that 
\[ \label{tz}
z_{\tilde \zeta} :=  {\tilde \zeta}(1+{\tilde \zeta}^*{\tilde \zeta})^{-1/2} \in\M(\B).
\]
Let $\K$ be a Hilbert space and let $\pi: \B\ra \rmB(\K)$ be a non-degenerate *-representation. 
As mentioned before, $\pi$ 
admits a unique extension $\pi: \M(\B)\ra \rmB(\K)$. 
By \eqref{pi-im} and the isomorphism \eqref{isopsi}, the $\pi$-image of $\zeta$ is 
$\pi(\zeta)= \pi(z_{\tilde \zeta}) \big(1- \pi(z_{\tilde \zeta})^*\pi(z_{\tilde \zeta})\big)^{-1/2}$. 
On the other hand, plugging $\pi(\zeta)$ into \eqref{zetatrans} yields
$z_{\pi(\zeta)}= \pi(z_{\tilde \zeta})$, 
so $\pi(\zeta)$ is uniquely determined by $ \pi(z_{\tilde \zeta})=z_{\pi(\tz)}$ 
(see the comment below Equation \eqref{pi-im}). 

Suppose that we are given  two representations $\pi_i: \B\ra \rmB(\K)$, $i=1,2$. Then the statement 
``$\zeta$ separates the representations of $\B$\hs'' means $\pi_1(\zeta)\neq \pi_2(\zeta)$ whenever $\pi_1\neq\pi_2$. 
By the previous discussion, this is equivalent to $\pi_1(z_{\tilde \zeta})=\pi_2(z_{\tilde \zeta})$
implies $\pi_1=\pi_2$.
 
 Our first aim is to show that 
 $|\pi_1(z_{\tilde \zeta})|=|\pi_2(z_{\tilde \zeta})|$ 
entails $\pi_1(f(|\tilde \zeta|))=\pi_2(f(|\tilde \zeta|))$ 
 for an appropriate class of continuous functions $f$ on $\spec(|\zeta|)$. 
 To begin with, observe that  
 $\pi(z_{\tilde \zeta})^*  \pi(z_{\tilde \zeta}) = \pi(z_{\tilde \zeta}^*z_{\tilde \zeta} ) 
 = \pi(|z_{\tilde \zeta}|^2)=  \pi(|z_{\tilde \zeta}|)^2$ gives 
 $ |\pi(z_{\tilde \zeta})|=\pi(|z_{\tilde \zeta}|)$. Here the fact that $z_{\tilde \zeta}\in \M(\B)$ 
 implies $z_{\tilde \zeta}^*, \;|z_{\tilde \zeta}|\in \M(\B)$ is used. 
 Next,  we invoke the 
 Gelfand representation to obtain an isometric embedding 
 \[   \label{iota}
 \iota : C_0(\spec(|\zeta|)) \lhra \B, \quad \iota(f):= f(|\tilde \zeta|). 
 \]
 Combining it with $\pi:\M(\B)\ra \rmB(\K)$ yields a representation  $\pi:  \iota(C_0(\spec(|\zeta|)))\ra \rmB(\K)$. 
 From \eqref{tz}, we see that $|z_{\tilde \zeta}| = |\tilde \zeta| (1+|{\tilde \zeta}|^2)^{-1/2}$. 
 The function $z(t):= t (1+t^2)^{-1/2}$ separates the points of the one-point compactification 
 $\spec(|\zeta|)\cup\{\infty\}$. By the Stone--Weiertrass Theorem, the functions $1$ and $z$ 
 generate $C(\spec(|\zeta|)\cup\{\infty\})$. Viewing $C_0(\spec(|\zeta|))$ as a subalgebra of 
 $C(\spec(|\zeta|)\cup\{\infty\})$ and extending $\pi\circ \iota$ to the multiplier algebra $\M(C_0(\spec(|\zeta|)))$, 
 it follows that $\pi\circ \iota: C_0(\spec(|\zeta|))\ra \rmB(\K)$ 
 is uniquely determined by its value  on $z \in \M(C_0(\spec(|\zeta|)))$, and so is its 
 extension to  $\M(C_0(\spec(|\zeta|)))$. Finally, it follows from \eqref{iota}, that 
 $\pi(\iota(z))=\pi( |\tilde \zeta| (1+|{\tilde \zeta}|^2)^{-1/2})=\pi(|z_{\tilde \zeta}|)= |\pi(z_{\tilde \zeta})|$. 
 Therefore, given  two representations $\pi_i: \B\ra \rmB(\K)$, $i=1,2$, we have 
 $|\pi_1(z_{\tilde \zeta})|=|\pi_2(z_{\tilde \zeta})|$ if and only if 
 $\pi_1 \circ \iota = \pi_2\circ \iota$, and the same is true for their 
 extensions to $\M(C_0(\spec(|\zeta|)))$. 
 
It still remains to show that $\pi_1(z_{\tilde \zeta}) = \pi_2(z_{\tilde \zeta})$ 
 implies $\pi_1(f(\tz)u^k) = \pi_2(f(\tz)u^k)$ for all $k\in\Z{\setminus}\{0\}$ and 
 $f\in C_0([0,\infty))$ with $f(0)=0$.  
 So consider again a representation $\pi$ lifted to the multiplier C*-algebra $\pi:\M(\B)\ra \rmB(\K)$, and  
 write $\pi(z_\tz)$ in its polar decomposition $\pi(z_\tz)= v|\pi(z_\tz)|$. If $u$ belonged to $\M(\B)$, 
 then it would suffice to show that $\pi(u)=v$, since then 
 $\pi(f(\tz)u^k) = \pi(f(\tz))\hs v^k$ would be uniquely determined by $v$ and 
 the arguments from the last paragraph. Unfortunately, $u\notin \M(\B)$ 
 since  $f(|\tz|)\in \B$ for  $f\in C_0([0,\infty))$ with $f(0)\neq 0$, but 
 $f(|\tz|)\hs u \notin \B$. For this reason, our proof is more complex, 
 involving the spectral theorem of self-adjoint operators. 
 
Let $F$ denote the unique  projection-valued measure on the
Borel $\sigma$-algebra $\Sigma([0,1])$ such that 
$|\pi(z_\tz)|=\int z \hs \dd F(z)$. 
Using $|\pi(z_\tz)|=\pi(|z_\tz|)$ 
from the paragraph below Equation \eqref{iota}, we get 
$|\pi(z_\tz)|=\pi(|\tilde \zeta (1+{\tilde \zeta}^*{\tilde \zeta})^{-1/2}|)=\pi(|\tilde \zeta| (1+|{\tilde \zeta}|^2)^{-1/2})= \pi(z_{|\tz|})$. 
By the definition of the $\pi$-image of operators affiliated  
to $\B$, we have 
$\pi(|\tz|)= \pi(z_{|\tz|}) (1- \pi(z_{|\tz|})^2)^{-1/2}$. 
Therefore we can write  $\pi(|\tz|) = \int z(1- z^2)^{-1/2}  \dd F(z)$. 
The function $z : [0,\infty]\ra [0,1]$, \,$z(\tau)=\frac{\tau}{\sqrt{1+\tau^2}}$ 
is a homeomorphism with inverse  
$\tau: [0,1]\ra [0,\infty]$, \,$\tau(z)=\frac{z}{\sqrt{1-z^2}}$. 
Let $E$ denote projection-valued measure on $\Sigma([0,\infty])$ given by the 
pull-back of $F$ under $z$, i.e., $E(M):=F(z(M))$ for all 
$M\in \Sigma([0,\infty])$. 
Then 
\[ \label{int}
\pi(|\tz|) = \int \tau\hs  \dd E(\tau)\quad \text{and}\quad
\pi(z_{|\tz|})   =    \int z(\tau) \hs \dd E(\tau)
=\int  \frac{\tau}{\sqrt{1+\tau^2}}\hs \dd E(\tau).  
 \]
 Observe also that $E(\{\infty\})=0$ since $F(\{1\})=0$.

We claim that 
$\pi(f(|\tz|))= \int f(\tau)\hs  \dd E(\tau)$ 
for all $f\in C([0,\infty])$. 
Note that $\pi(f(|\tz|))$ is well defined for such an $f$ since then $f(|\tz|)\in\M(\B)$. 
As explained above, the function $z \in C([0,\infty])$ separates the points 
of $[0,\infty]$ so that for each $f\in C([0,\infty])$ there exists a sequence 
of polynomials $\{p_n\}_{n\in\N}$ such that $\underset{n\ra\infty}{\lim} p_n(z)=f$ 
in the supremum norm. Using uniform convergence, $z_{|\tz|}=z(|\tz|)$  
and the second equation in \eqref{int}, we get 
$$
\pi(f(|\tz|))=\hsp \underset{n\ra\infty}{\lim} \pi(p_n(z(|\tz|) ))= \hsp
\underset{n\ra\infty}{\lim} p_n(\pi(z_{|\tz|} ))=\hsp
\underset{n\ra\infty}{\lim} \hsp \int \! p_n(z(\tau))  \dd E(\tau)
=\! \int \! f(\tau)  \dd E(\tau), 
$$
which proves our claim.  Combining it with  \eqref{int} gives  $\pi(f(|\tz|)) = f(\pi(|\tz|))$. 
 
 Now let $\{\varphi_n\}_{n\in\N}\subset  C_0((0,\infty))$ 
 be a sequence of functions of rapid decay, i.e., 
 \[ \label{varphin}
 \sup\{ \, t^k |\varphi_n(t)|\, :\, t\in (0,\infty)\,\} < \infty \ \  \text{for all}\ \ k\in\Z, 
 \]
 such that 
 $ 0\leq \varphi_1\leq \varphi_2\leq \ldots \leq 1$  \,and \,$\underset{n\ra\infty}{\lim}\hs \varphi_n(t) =1$ 
 for all $t\in (0,\infty)$. 
 With the obvious extension to continuous function on $[0,\infty]$, it follows that 
 $$
 \underset{n\ra\infty}{\lim} \pi(\varphi_n(|\tz|)) = \underset{n\ra\infty}{\lim} \varphi_n(\pi(|\tz|)) = E((0,\infty)) 
 $$ 
 in the operator weak topology (even in the  operator strong topology). 
 
 Let us recall that $\tz = u |\tz|$ defines a $q$-normal operator on $\ker(|\zeta|)^\perp$ so that, 
 by Proposition \ref{qnormal}, 
 $u f(|\tz|) u^* = f(q |\tz|)$ for every Borel function $f$ on $[0,\infty]$. 
 Furthermore, 
 $$
 z_{\tilde \zeta} =  {\tilde \zeta}(1+{\tilde \zeta}^*{\tilde \zeta})^{-1/2}= u |\tz|(1+|{\tilde \zeta}|^2)^{-1/2} = 
 u\hs z(|\tz|) = z(q|\tz|) \hs u\hs .
 $$ 
 As a consequence, for all $f\in C_0([0,\infty))$ and $k\in\N$, 
 \begin{align} \nonumber 
 \varphi_n(|\tz|) f(|\tz|) u^k &= 
 \varphi_n(|\tz|)\Big(  \prod_{j=1}^k z(q^j |\tz|)^{-1}  \Big) f(|\tz|) \big(u\, z(|\tz|) \big)^k  \\
 &=  \varphi_n(|\tz|)\Big(  \prod_{j=1}^k z(q^j |\tz|)^{-1}  \Big) f(|\tz|)\hs z_{\tilde \zeta}^k\,,   \label{ptrick}\\
 \nonumber 
 \varphi_n(|\tz|) f(|\tz|) u^{-k} &= 
 \varphi_n(|\tz|)\Big(  \prod_{j=0}^{k-1}z(q^{-j} |\tz|)^{-1}  \Big) f(|\tz|) \big(z(|\tz|)\hs u^* \big)^k  \\
 &=  \varphi_n(|\tz|)\Big(  \prod_{j=0}^{k-1}z(q^{-j} |\tz|)^{-1}  \Big) f(|\tz|)\hs  z_{\tilde \zeta}^{*k}.\label{ptrick-}
 \end{align}
  From \eqref{varphin}, it follows that the function 
 $(0,t)\ni t\mapsto \varphi_n(t)\Big(\overset{k}{\underset{j=1}{\prod}}z(q^j t)^{-1}  \Big)$
 belongs to $C_0((0,\infty))$ and can therefore also be considered as an element of $C_0([0,\infty))$. 
 The same holds for the function 
 $(0,t)\ni t\mapsto \varphi_n(t)\Big(\overset{k-1}{\underset{j=0}{\prod}}z(q^{-j} t)^{-1}  \Big)$.
 
 To finish the proof of (\ref{Tb}), assume that $\pi_i: \B\ra \rmB(\K)$, \,$i=1,2$, are two representations satisfying 
 $\pi_1(z_{\tilde \zeta}) = \pi_2(z_{\tilde \zeta})$. 
 In particular, we have $|\pi_1(z_{\tilde \zeta})| = |\pi_2(z_{\tilde \zeta})|$.
 It has already be shown that then $\pi_1(f(|\tilde \zeta|))=\pi_2(f(|\tilde \zeta|))$ for all 
 $f\in C([0,\infty])$. This also implies $\pi_1(|\tz|)=\pi_2(|\tz|)$ by the definition of the $\pi$-image of $|\tz|$. 
 As in \eqref{int}, write  $\pi_i(|\tz|) = \int \tau\hs  \dd E(\tau)$. 
 Using $\pi_i(f(|\tilde \zeta|))=f(\pi_i(|\tilde \zeta|))$, it follows that 
 $E(\{0\}) \pi_i(f(|\tilde \zeta|)) = f(0) E(\{0\})$. As a consequence, 
 $\pi_i(f(|\tilde \zeta|)) = E((0,\infty))\hs\pi_i(f(|\tilde \zeta|))$ for all $f\in C_0([0,\infty))$ with $f(0)=0$. 
 Given such an $f$, we compute for all $k\in\N$  by taking operator weak limits 
 \begin{align}\nonumber
 \pi_1(f(|\tilde \zeta|) u^k)
 &= E((0,\infty))\hs\pi_1(f(|\tilde \zeta|) u^k) 
 =\underset{n\ra\infty}{\lim} \pi_1(\varphi_n(|\tz|))\hs\pi_1(f(|\tilde \zeta|) u^k) \nonumber\\
 &=\underset{n\ra\infty}{\lim} \pi_1\bigg(
 \varphi_n(|\tz|)\Big(  \prod_{j=1}^k z(q^j |\tz|)^{-1}  \Big) f(|\tz|) \hs z_{\tilde \zeta}^k \bigg) \nonumber\\
 &=\underset{n\ra\infty}{\lim} \pi_1\bigg(
 \varphi_n(|\tz|)\Big(  \prod_{j=1}^k z(q^j |\tz|)^{-1}  \Big) f(|\tz|)\bigg) \pi_1( z_{\tilde \zeta})^k \nonumber\\
 &=\underset{n\ra\infty}{\lim} \pi_2\bigg(
 \varphi_n(|\tz|)\Big(  \prod_{j=1}^k z(q^j |\tz|)^{-1}  \Big) f(|\tz|)\bigg) \pi_2( z_{\tilde \zeta})^k \nonumber\\
 &=\pi_2(f(|\tilde \zeta|) u^k). \label{uk}
 \end{align}
 Here, we applied \eqref{ptrick} in the passage from the first to the second line, and used  
 the property that $\pi$ defines a representation of $\M(\B)$ in the next line. 
 The crucial step from the third to the fourth line 
 follows from $\pi_1(z_{\tilde \zeta})= \pi_2(z_{\tilde \zeta})$ by invoking the assumption and from 
 the previously proven fact that 
 $\pi_1(g(|\tilde \zeta|))=\pi_2(g(|\tilde \zeta|))$ for all $g\in C([0,\infty])$. 
 The last equality is obtained by repeating all steps performed for $\pi_1$  backwards. 
 
 The same arguments with \eqref{ptrick} replaced by \eqref{ptrick-} show that \eqref{uk} 
 holds also for $k\in\Z$, \,$k<0$. 
To sum up, we have shown for all $k\in\Z$ and all 
 $f_k \in C_0([0,\infty))$ satisfying $f_k(0)=0$ if $k\neq 0$
that $\pi_1(z_{\tilde \zeta}) = \pi_2(z_{\tilde \zeta})$ implies 
 $\pi_1(f_k(|\tilde \zeta|) u^k) = \pi_2(f_k(|\tilde \zeta|) u^k)$ . 
 Since these elements generate $\B$, we finally conclude that $\pi_1=\pi_2$. 
 \end{proof} 

We remark that if we had considered the C*-algebra generated by 
$|\zeta|$ and $u$, where $\zeta = u |\zeta|$, then the generated C*-algebra would 
be $C_0( X)\rtimes \Z$ by the simple argument given in the second paragraph 
following Equation \eqref{iota}. 
A more general construction with $|\zeta|$ having discrete spectrum 
can be found in \cite{PS}.

Assume now that $Z$ is a $q$-normal operator such that 
$\spec(Z)=[0,\infty)$. Then 
$$
C_0^*(Z,Z^*):=\|\cdot\|\text{-}\mathrm{cls}\Big\{ 
\sum_{\text{{\rm finite}}} f_k \hs U^k\in C_0( [0,\infty))\rtimes \Z \;:\; k\in\Z,\ \,f_k(0)=0\ \, 
\text{for all}\ \,k\neq 0 
\Big\} 
$$ 
can be viewed as 
an universal object of (the category of) C*-algebras generated by 
$q$-normal operators since 
$$
C^*_0(Z,Z^*) \ni \sum_{\text{{\rm finite}}} f_k \hs U^k \;
\longmapsto\; \sum_{\text{{\rm finite}}} f_k \!\!\upharpoonright_{X} U^k  \in C^*_0(\zeta,\zeta^*) 
$$
yields a well-defined *-homorphism for all 
C*-algebras $C^*_0(\zeta,\zeta^*)$ from \eqref{Czz}, see the 
proof of 
Proposition \ref{propX}. 

For a geometric interpretation of $C^*_0(Z,Z^*)$, observe that 
the definitions of the crossed product C*-algebra 
$C_0( [0,\infty))\rtimes \Z$  and of $C^*_0(Z,Z^*)$ still 
make sense if we set $q=1$. 
In this case, $\alpha_q = \id$ and 
both algebras become commutative.  
If $f_k\in C_0([0,\infty))$ and $f_k(0)=0$ for $k\neq 0$, then 
$f_kU^k$ can be viewed as a function in $C_0(\C)$ 
by using Euler's formula $Z=|Z|\hs \e^{\im \theta}$ 
and assigning 
$|Z|\hs \e^{\im \theta} \mapsto f_k(|Z|)\hs \e^{\im \theta k}$. 
Note that  $f_k(0) =0$ if $k\neq 0$ is crucial since $|Z|=0$ 
corresponds to the unique point $Z=0$ 
so that the function must be independent from~$\theta$. 
Furthermore, one easily checks that the algebra of functions 
\mbox{$|Z|\hs \e^{\im \theta} \mapsto \underset{\text{finite}}{\sum} f_k(|Z|)\hs \e^{\im \theta k}$}
separates the points 
of $\C\cup \{\infty\}$. Here, it is crucial to include functions 
$f_0\in C_0([0,\infty))$ satisfying $f(0)\neq 0$ since otherwise 
the points $0$ and $\infty$ could not be distinguished.  
By the Stone--Weierstrass theorem, the algebra of 
functions just defined generates $C_0(\C)$. 
Finally, since $C_0(\C)$ is commutative, 
the universal norm coincides with the supremum norm. 
Therefore, we obtain 
$C_0^*(Z,Z^*)=C_0(\C)$ for  $q=1$. 
This motivates the following definition. 
\begin{defn} \label{def}
We say that 
$$
C_0(\C_q):=\|\cdot\|\text{-}\mathrm{cls}\Big\{ 
\sum_{k=M}^N f_k \hs U^k\in C_0( [0,\infty))\rtimes \Z \,:\, f_k(0)=0\ \, 
\text{for all}\ \,k\neq 0 
\Big\}
$$ 
is the C*-algebra of continuous functions vanishing 
at infinity on the quantum complex plane $\Cq$. 
Its unitization  
$$
C(\mathrm{S}^2_q):= C_0(\C_q)\oplus \C
$$
is called the quantum sphere generated by $\Cq$. 
\end{defn}

\section{Final remarks} 
The calculations in \cite{CW} 
 show that $C(\mathrm{S}^2_q)$ has the same $K$-groups as the classical 2-sphere, 
 that is, $K_0(C(\mathrm{S}^2_q))\cong \Z\oplus \Z$ and $K_1(C(\mathrm{S}^2_q))\cong 0$. 
 It is expected that the non-trivial part of $K_0(C(\mathrm{S}^2_q))$ can be described by analogues 
 of the classical Bott projections 
 $$
   P_{n} :=  \left(
\begin{array}{c}
1  \\
\zeta^{*n}
\end{array}\right) \mbox{$\frac{1}{1+q^{n(n+1)}|\zeta|^{2n}}$}\left(\begin{array}{cc} 1 &\hsp \zeta^{n}\end{array}\right) 
= \left(
\begin{array}{ll}
\mbox{$\frac{1}{1+q^{n(n+1)}|\zeta|^{2n}}$} & \mbox{$\frac{1}{1+q^{n(n+1)}|\zeta|^{2n}}$}\, \zeta^{n}\\[8pt]
\mbox{$\frac{1}{1+q^{-n(n-1)}|\zeta|^{2n}}$} \hs \zeta^{*n} & \mbox{$\frac{q^{-n(n-1)}|\zeta|^{2n}}{1+q^{-n(n-1)}\,|\zeta|^{2n}}$} 
\end{array}\hsp\right)\hsp ,\\[8pt]
 $$
 $$
   P_{-n} :=  \left(
\begin{array}{c}
1  \\
\zeta^{n}
\end{array}\right) \mbox{$\frac{1}{1+q^{-n(n-1)}|\zeta|^{2n}}$}\left(\begin{array}{cc} 1 &\hsp \zeta^{*n}\end{array}\right) 
= \left(
\begin{array}{ll}
\mbox{$\frac{1}{1+q^{-n(n-1)}|\zeta|^{2n}}$} & \mbox{$\frac{1}{1+q^{-n(n-1)}|\zeta|^{2n}}$}\, \zeta^{*n}\\[8pt]
\mbox{$\frac{1}{1+q^{n(n+1)}|\zeta|^{2n}}$} \, \zeta^{n} & \mbox{$\frac{q^{n(n+1)}|\zeta|^{2n}}{1+q^{n(n+1)}|\zeta|^{2n}}$} 
\end{array}\hsp \right)\hsp ,
 $$
 and that the pairing with the generators of the  $K$-homology group  $K^0(C(\mathrm{S}^2_q))\cong \Z\oplus \Z$ 
 computes the rank  (equal to  $1$) and 
 the "winding number" $\pm n$ of  projective module given by $P_{\pm n}$, $n\in\N$. 
 
 \section*{Acknowledgments} 
 We are grateful to the referee for useful comments. 
 The second author thanks the organizers of the conference "Operator Algebras and Quantum Groups - 
 a conference in honour of S. L. Woronowicz's seventieth birthday" for making his participation  possible. 
This work was carried out with partial financial support from the research grant PIRSES-GA-2008-230836.

\end{document}